\documentclass[reqno, 12pt]{amsart}
\usepackage[letterpaper,hmargin=1in,vmargin=1in]{geometry}
\usepackage{amsmath,amssymb,amsthm,amsaddr}
\usepackage{verbatim,epsfig,graphics}
\newtheorem{theorem}{Theorem}
\newtheorem{proposition}[theorem]{Proposition}

\newtheorem{lemma}[theorem]{Lemma}

\begin{document}
\title{The determinant on flat conic surfaces with excision of disks}
\author{David A. Sher}
\address{McGill University/ Centre de Recherches Math\'ematiques}
\email{david.sher@mail.mcgill.ca}

\begin{abstract} Let $M$ be a surface with conical singularities, and consider a family of surfaces $M_{\epsilon}$ obtained from $M$ by removing disks of radius $\epsilon$ around a subset of the conical singularities. Such families arise naturally in the study of the moduli space of flat metrics on higher-genus surfaces with boundary. In particular, they have been used by Khuri to prove that the determinant of the Laplacian is not a proper map on this moduli space when the genus $p\geq 1$ \cite{kh}. Khuri's work is closely related to the isospectral compactness results of Osgood, Phillips, and Sarnak. Our main theorem is an asymptotic formula for the determinant of $M_\epsilon$ as $\epsilon$ approaches zero up to terms which vanish in the limit. The proof uses the determinant gluing formula of Burghelea, Friedlander, and Kappeler along with an observation of Wentworth on the asymptotics of Dirichlet-to-Neumann operators. We then apply this theorem to extend and sharpen the results of Khuri.
\end{abstract}

\maketitle

\section{Introduction}

\renewcommand{\thefootnote}{\fnsymbol{footnote}} 
\footnotetext{\emph{Mathematics Subject Classification:} 58J50, 58J52.}     
\renewcommand{\thefootnote}{\arabic{footnote}} 

Let $(M,g)$ be a compact connected surface with finitely many
isolated conical singularities $P_{i}$ of
cross-section $\alpha_{i}S^{1}_\theta$. In particular, 
in a neigborhood of each $P_{i}$, we assume that 
$(M,g)$ is isometric to $(0,3/2)_r\times S^1_\theta$ with the metric
\[dr^{2}+r^{2}\alpha_{i}^{2}d\theta^{2}\]
for some constant $\alpha_{i}>0$.
We say that
$2\pi\alpha_{i}$ are the \emph{cone angles}; note that when $\alpha_{i}=1$,
the conic singularity at $P_{i}$ becomes smooth.

Choose a nonempty subset $\{P_1, P_2, \ldots, P_k\}$ of the conical singularities. For each $i$ with $1\leq i\leq k$, fix a real number $b_i\in(0,1]$. For each $\epsilon\in (0,1]$, let $\Gamma_{\epsilon,i}$ be the
curve $r=b_i\epsilon$, and let $\Gamma_\epsilon$ be the union of all $\Gamma_{\epsilon,i}$. 
The region inside $\Gamma_{\epsilon,i}$ is a cone
of angle $2\pi\alpha_i$ and length $b_i\epsilon$, which we call 
$C_{\alpha_i,b_i\epsilon}$. We then let $M_{\epsilon}$ be the connected manifold
with boundary obtained by removing all of the $C_{\alpha_i,b_i\epsilon}$ from $M$; its boundary is precisely $\Gamma_\epsilon$. Finally, let
\[\beta=\frac{1}{k}\sum_{j=1}^kb_1\ldots\hat b_j\ldots b_k,\]
where the hat symbol indicates that the $j$th factor is omitted from the product. 

Our main theorem is the following asymptotic expansion:
\begin{theorem}\label{truncated} Let
$M_{\epsilon}$ be as defined above.
As $\epsilon\rightarrow 0$, we have
\[\log\det \Delta_{M_{\epsilon}}=(k-1+\sum_{i=1}^k\frac{1}{6}(\alpha_i+\frac{1}{\alpha_i}))\log\epsilon+(k-1)\log\log(1/\epsilon)+\log\det\Delta_M\]
\begin{equation}\label{second}-\sum_{i=1}^k
\log\det\Delta_{C_{\alpha_i,1}}
+k\log 2-\log V+\sum_{i=1}^k\frac{1}{6}(\alpha_i+\frac{1}{\alpha_i})\log b_i+\log\beta+o(1).
\end{equation}
Here we are imposing Dirichlet boundary conditions at all boundaries and the Friedrichs extension at all conic singularities. We also ignore the zero eigenvalue of $\Delta_M$ for the purposes of defining its determinant.
\end{theorem}
Each term in the expansion except $\log\det\Delta_M$, $\log\det\Delta_{M_\epsilon}$, and $\log V$ may be computed explicitly as a function of the cone angles $\alpha_i$ and the constants $b_i$. In particular, Spreafico has computed a formula for
$\log\det\Delta_{C_{\alpha,1}}$ \cite{s}. The error is only $o(1)$ in general, but in the $k=1$ case there is an explicit upper bound:
\begin{theorem}\label{bettererror} When $k=1$, $b_1=1$, and $\epsilon\leq 2^{-\alpha_1}$, the error in (\ref{second}) is bounded by $6\epsilon^{1/\alpha_1}.$
\end{theorem}

\subsection{Relation to isospectral compactness}
In a series of papers \cite{ops1,ops2,ops3} from the 1980s, Osgood, Phillips, and Sarnak investigate sets of isospectral metrics on surfaces. Their key result is that in the setting of closed surfaces, any set of isospectral metrics is compact in the natural $C^{\infty}$ topology. The proof makes extensive use of the determinant of the Laplacian. First, in \cite{ops1}, Osgood, Phillips, and Sarnak prove that within any conformal class of metrics (normalized so that the area is constant) on a closed surface, the determinant of the Laplacian is maximized by the constant curvature metric. They then show that on the moduli space of constant curvature metrics, the log-determinant of the Laplacian is a proper map. These two facts are key to the proof of isospectral compactness. In \cite{ops3}, an analogous result is proved for isospectral sets of planar domains (normalized so that the boundary length is constant), with the usual flat Euclidean metric and Dirichlet boundary conditions.

Khuri, in \cite{kht,kh}, addressed the natural question of whether a similar result holds for isospectral sets of flat metrics on a topological surface $\Sigma_{p,n}$, which is a surface of genus $p$ with $n$ disks removed. The setting $p=0$ is precisely the case of planar domains and $n=0$ is the case of closed surfaces, so Khuri only considered $np\geq 1$. For $np\geq 1$, Khuri showed by constructing counterexamples that the log of the determinant is not a proper map on the appropriate moduli space of flat metrics \cite{kh}. The counterexamples are closely related to our construction of $M_{\epsilon}$; namely, Khuri took surfaces $(M,g)$ of genus $p$ with $n$ isolated conic singularities and removed a disk of radius $\epsilon$ around each singular point, then let $\epsilon$ go to zero. After normalizing so that the boundary length is constant, Khuri showed that the resulting surfaces $(M_\epsilon,\hat g)$ approach the boundary of moduli space but have $\log\det\Delta_{(M_\epsilon,\hat g)}$ bounded below. Therefore, the Osgood-Phillips-Sarnak approach does not work for $np\geq 1$ \cite{kh}. On the other hand, using a comparison of two different moduli spaces, Y.-H. Kim has recently proved isospectral compactness of sets of flat metrics on $\Sigma_{p,n}$ in this $np\geq 1$ setting \cite{kim}. The problem remains open, for all $p$ and all $n\geq 1$, if the isospectral metrics are not assumed to be flat.

Our goal is to study the behavior of the determinant of the Laplacian as a function on the space of metrics on $\Sigma_{p,n}$. Based on the results of Khuri and Kim, we expect that this work may have further applications to the isospectral problem and related questions. As a first step, we apply Theorem \ref{truncated} to sharpen the results of Khuri in \cite{kht,kh}. Let $(M,g)$ be a fixed surface of genus $p$, where $g$ is a flat
conical metric with $n$ conical singularities, and
let $V$ be the volume of $M$.
Let $M_{\epsilon}$ be the surface obtained from $M$ by removing a cone
of radius $\epsilon$ around each of the $n$ cone points. We can then apply
Theorem \ref{truncated} to $M_{\epsilon}$, where $k=n$ and $b_i=1$ for all $i$.
However,
Khuri normalizes so that the geodesic curvature
is constant and equal to $-1$ on the boundary, which means multiplying
the metric on $M_{\epsilon}$ by $\epsilon^{-2}$; let $\hat g=\epsilon^{-2}g$.
By Proposition \ref{scaling}, \begin{equation}
\label{scalingthing}\log\det\Delta_{(M_{\epsilon},\hat g)}=
\log\det\Delta_{(M_{\epsilon},g)}+2\zeta_{M_{\epsilon}}(0)\log\epsilon.\end{equation}
On the other hand, since $M_\epsilon$ is a smooth surface with boundary, the McKean-Singer heat asymptotics of \cite{ms} (a good exposition may also be found in \cite{r}) imply that
\begin{equation}\label{zetathing}
\zeta_{M_{\epsilon}}(0)=\frac{1}{6}\chi(M_{\epsilon})=\frac{1}{6}
(\chi(M)-n)=\frac{1}{6}(2-2p-n).\end{equation}
Moreover, by the Gauss-Bonnet theorem, the sum of all $\alpha_{i}$ must
be equal to $2p+n-2$ \cite{kht}.
Combining Theorem \ref{truncated} with (\ref{scaling}), (\ref{zetathing}),
and the Gauss-Bonnet theorem, we have shown:
\begin{proposition}\label{modtrunc} The determinant of the Laplacian of Khuri's metrics
$(M_{\epsilon},\hat g)$ has an expansion as $\epsilon\rightarrow 0$:
\[\log\det\Delta_{(M_{\epsilon},\hat g)}=
\frac{1}{6}(2-2p-n+\sum_{i=1}^{n}(\frac{1}{\alpha_{i}}))\log\epsilon
+(n-1)\log\log 1/\epsilon\]
\begin{equation}\label{khuriexp}+\log\det \Delta_{(M,g)}
-\sum_{i=1}^{n}\log\det\Delta_{C_{\alpha_{i},1}}+n\log 2
-\log V+o(1).
\end{equation}\end{proposition}
In the $n=1$ case, we also obtain a better error estimate from Theorem \ref{bettererror}. As a consequence of Proposition \ref{modtrunc}, we conclude that
\begin{proposition} The log-determinant $\log\det\Delta_{(M_\epsilon,\hat g)}$ approaches $-\infty$ as $\epsilon\rightarrow 0$ iff
\begin{equation}\label{condition}\sum_{i=1}^n\frac{1}{\alpha_i}>2p+n-2.\end{equation}
\end{proposition}
Notice that as any individual $\alpha_i$ goes to zero, the condition (\ref{condition}) is eventually satisfied; in particular, if even one of the cone angles is small enough, the determinant of $M_\epsilon$ will go to $-\infty$ as $\epsilon\rightarrow 0$. This observation suggests that it might be fruitful to analyze the behavior of the expansion in Proposition \ref{modtrunc} as one or more of the cone angles $\alpha_i$ degenerate to zero. Such an analysis would likely require a joint asymptotic expansion in $\epsilon$ and $\alpha_i$, which could be difficult to prove, but the reward would be a much better understanding of the determinant on moduli space.

\subsection{Examples} We now examine some particular cases.
First we let $p=1$ and $n=1$;
in this setting, Khuri proved that $\log\det\Delta_{(M_{\epsilon},\hat g)}$
is bounded below (Equation 4.3.15 of \cite{kht}). 
Note that by the Gauss-Bonnet theorem, the only cone
angle $\alpha$ must be 1. Plugging $p=1$, $n=1$, and $\alpha=1$
into (\ref{khuriexp}),
we see that the coefficients of $\log\epsilon$
and $\log\log 1/\epsilon$ vanish, and hence the log determinants actually converge
to a constant. This sharpens Khuri's result.

Next we let $p=1$ and $n\geq 2$. In this case, Khuri showed (Equation
4.3.19 of \cite{kht}) that if all cone angles $\alpha_{i}$ are equal to $1$,
\[\log\det\Delta_{(M_{\epsilon},\hat g)}\geq (n-1)\log\log 1/\epsilon+C.\]
The coefficient of $\log\epsilon$ in our formula again vanishes, and the coefficient of
$\log\log 1/\epsilon$ is precisely $n-1$, which shows that Khuri's inequality is sharp and also identifies the constant $C$. On the other hand, when the cone angles $\alpha_{i}$ are not all equal to 1,
they must still sum to $n$, so $\sum_{i=1}^n\alpha_i^{-1}>n$, and the coefficient of $\log\epsilon$ is
positive. In this case, the log determinants do in fact decrease to
$-\infty$ as $\epsilon\rightarrow 0$; Khuri did not consider this case.

Finally, we let $p\geq 2$ and $n=1$. Here the cone angle $\alpha_{i}$ must
be $2p-1$. The coeficient of the $\log|\log\epsilon|$ term vanishes,
but the coefficient of $\log\epsilon$ is precisely
\[\frac{1}{6}(1-2p+\frac{1}{2p-1}).\]
When $p\geq 2$, this is positive, so the log determinants increase to $\infty$
as $\epsilon\rightarrow 0$. It must be noted that the leading-order coefficient of $\log\epsilon$ is inconsistent with Equation 4.3.22 in \cite{kht}. However, the calculation
in \cite{kht} is not correct, as the asymptotic behavior of the determinant
of an annulus is computed with the Euclidean metric when it should be computed with the conic metric.

\subsection{Organization}

The proof of this theorem proceeds in two steps. First we use the gluing formula
of Burghelea, Friedlander, and Kappeler \cite{bfk} to break up the manifold
$M$ along $\Gamma_{\epsilon}$ and decompose $\log\det\Delta_{M}$ as a sum
of the log-determinants of the individual pieces. The only non-explicit term in
the BFK formula is the log-determinant of the 
Neumann jump operator on $\Gamma_{\epsilon}$. 
In the second step, we compute the asymptotics of this Neumann jump operator
as $\epsilon\rightarrow 0$. This computation is based closely on an
observation of Wentworth \cite{w}. Putting these two steps together
gives Theorem \ref{truncated}. The combination of Wentworth's observation and the BFK gluing formula is a natural one which has proven useful in other investigations of the determinant on moduli space \cite{ko}.

\subsection{Acknowledgements}
This paper is a generalization of the final chapter of my Stanford Ph.D. thesis \cite{sh}. I am deeply grateful to my advisor Rafe Mazzeo for all his help and support. Additionally, this paper would not have been written without a discussion with Richard Wentworth; I would like to thank him and also the other organizers of the conference 'Analysis, Geometry, and Surfaces' at Autrans in March 2011. Finally, I would like to thank Alexey Kokotov and Andras Vasy for helpful comments and bug-spotting, as well as the ARCS foundation for support in 2011-2012.

\section{Determinant gluing formula}

In this section, we recall the Burghelea-Friedlander-Kappeler gluing formula and apply it to our problem.
For any $\epsilon>0$, we define an operator $R_{\epsilon}$ on
$\Gamma_{\epsilon}$, following \cite{bfk}: 
if $f$ is a function on $\Gamma_{\epsilon}$,
we let $u_{+}$ be the solution of $\Delta_{M_{\epsilon}}u_{+}=0$
with $u_{+}|_{\Gamma_{\epsilon}}=f$.
Similarly, let $u_{-}$ be the harmonic
function on the union of the $C_{\alpha_i,b_i\epsilon}$ with boundary data
$f$. As always, we require $u_{-}$ and $u_{+}$ to be in the
Friedrichs domain at all conic singularities. 
Then in a neighborhood of $\Gamma_{\epsilon}$,
$g=\partial_{r}u_{-}-\partial_{r}u_{+}$ is well-defined; we let
$R_{\epsilon}f$ be the restriction of $g$ to $\Gamma_{\epsilon}$.
In fact, $R_{\epsilon}$ is simply the sum of the Dirichlet-to-Neumann
operators on $M_{\epsilon}$ and $C_{\alpha_i,b_i\epsilon}$, and we call it
the \emph{Neumann jump operator} for $\Gamma_{\epsilon}$.

It is well-known
that $R_{\epsilon}$ is an elliptic pseudodifferential operator of order 1,
and that it is possible to define a zeta function and determinant of $R_{\epsilon}$
in the usual fashion \cite{bfk}. In this paper, when we define determinants, we always leave out the zero eigenvalues. For $R_\epsilon$, there is precisely one:
\begin{lemma} The kernel of $R_{\epsilon}$ is equal to the set of
globally constant functions on $\Gamma_{\epsilon}$.
\end{lemma}
\begin{proof} Suppose that $f$ is in the kernel of $R_{\epsilon}$, and
construct $u_{+}$ and $u_{-}$ as above. Since $\partial_{r}u_{-}=
\partial_{r}u_{+}$ on $\Gamma_{\epsilon}$, and $u_{+}=u_{-}$ on $\Gamma_{\epsilon}$,
$u_{+}$ and $u_{-}$ glue together to define a $C^{1}$ function $u$ on
$M$ which solves $\Delta_{M}u=0$ weakly. Since all functions are in the 
Friedrichs domain at the cone points, $u$ is bounded. 
By elliptic regularity, $u$
is in fact a smooth bounded harmonic function on $M$, and therefore must be
constant. \end{proof}

We now have the following gluing formula, due to Burghelea, Friedlander, and
Kappeler \cite{bfk} in the smooth setting and to Loya, McDonald and Park 
\cite{lmp} in the setting of manifolds with conical singularities:
\begin{proposition}\cite{bfk} 
Let $V$ be the volume of $M$. Then
\begin{equation}\label{bfkgluing}\log\det \Delta_{M}=
\log\det\Delta_{M_{\epsilon}}+\sum_{i=1}^k
\log\det\Delta_{C_{\alpha_i,b_i\epsilon}}
+\log V-\sum_{i=1}^k\log 2\pi\alpha_ib_i\epsilon +\log\det R_{\epsilon}.
\end{equation}
\end{proposition}
Note that
the length of $\Gamma_{\epsilon}$ is exactly $\sum_{i=1}^k 2\pi\alpha_ib_i\epsilon$.
Rearranging (\ref{bfkgluing}) gives us a formula for $\log\det\Delta_{M_{\epsilon}}$
in terms of the other undetermined quantities:
\[\log\det \Delta_{M_{\epsilon}}=
\log\det\Delta_{M}-\sum_{i=1}^k
\log\det\Delta_{C_{\alpha_i,b_i\epsilon}}
-\log V\]
\begin{equation}\label{bfkgluingtwo}+\sum_{i=1}^k\log 2\pi\alpha_ib_i+k\log\epsilon-\log\det R_{\epsilon}.
\end{equation}
To simplify the formula, we make a scaling observation:
\begin{proposition}\label{scaling}
Let $(\Omega,g)$ be any compact Riemannian surface, with or without boundary and
with or without isolated conic singularities. Then
\[\log\det\Delta_{\Omega,\epsilon^{2}g}=\log\det\Delta_{\Omega,g}-2
\zeta_{\Omega,g}(0)\log\epsilon.\] \end{proposition}
\begin{proof}If we scale the metric
$g$ by $\epsilon^{2}$, we scale the eigenvalues of the Laplacian by $\epsilon^{-2}$.
The proposition then follows from taking the derivative of
$\epsilon^{2s}\zeta_{\Omega,g}(0)$ at $s=0$. 
\end{proof}Note that, crucially, $\zeta_{\Omega,g}(s)$ has no pole at $s=0$; this is not true for manifolds with conic singularities in general, but is true in two dimensions, where the cross-section is a circle (see \cite{kh} for a detailed explanation). If a pole were present, an analogue of Proposition \ref{scaling} would still hold (using the Laurent series definition of the determinant), but with extra terms.

The special value $\zeta_{\Omega,g}(0)$ is known in many cases. 
In particular a result of Spreafico \cite{s} states that
\[\zeta_{C_{\alpha,1}}(0)=\frac{1}{12}(\alpha+\frac{1}{\alpha}). \]
Combining this with (\ref{bfkgluingtwo}) and rearranging, we obtain
\[\log\det \Delta_{M_{\epsilon}}=(k+\sum_{i=1}^k\frac{1}{6}(\alpha_i+\frac{1}{\alpha_i}))\log\epsilon+\log\det\Delta_M-\sum_{i=1}^k
\log\det\Delta_{C_{\alpha_i,1}}\]
\begin{equation}\label{intermediate}+\sum_{i=1}^k\frac{1}{6}(\alpha_i+\frac{1}{\alpha_i})\log b_i
+\sum_{i=1}^k\log 2\pi\alpha_ib_i-\log\det R_{\epsilon}-\log V.
\end{equation}

\section{The Neumann jump operator}

We now prove the following asymptotic
expansion for $\log\det R_{\epsilon}$ as $\epsilon$ goes to zero; combining it with (\ref{intermediate}) yields Theorem \ref{truncated}.

\begin{theorem}\label{asymplogdet} As $\epsilon\rightarrow 0$, 
\begin{equation}\label{njoasymp}\log\det R_\epsilon=\log\epsilon-(k-1)\log\log(1/\epsilon)+\sum_{i=1}^k\log\pi b_i\alpha_i-\log\beta+o(1).\end{equation} Moreover, when $k=1$, $b_1=1$, and $\epsilon\leq 2^{-\alpha_1}$, the error in (\ref{njoasymp}) is bounded by $6\epsilon^{1/\alpha_1}$.
 \end{theorem}

The proof of this theorem in the special case $k=1$, $\alpha_1=1$, and $b_1=1$ is due to Wentworth \cite{w}, albeit without the error bound; we adapt the proof given there.
\subsection{Computation of $R_{\epsilon}$}
Identifying $\Gamma_{\epsilon}$ with $\oplus_{i=1}^kS^{1}$,
we view $R_{\epsilon}$ as an operator on $L^2(\oplus_{i=1}^kS^{1})$.
The space $L^2(\oplus_{i=1}^kS^{1})$ admits an orthogonal decomposition
\begin{equation}\label{orthodecomp}L^2(\oplus_{i=1}^kS^{1})=\mathcal K\oplus\mathcal C\oplus L_0^2.\end{equation}
Here $\mathcal K$ is the one-dimensional space of globally constant functions, $\mathcal C$ is the $k-1$-dimensional space of locally constant functions which integrate to zero, and $L_0^2$ is the subspace of $L^2(\oplus_{i=1}^kS^{1})$ spanned by the non-constant eigenfunctions on each component. Let $\Pi_{\mathcal K}$, $\Pi_{\mathcal C}$, and $\Pi_0$ be the orthogonal projections onto each summand in (\ref{orthodecomp}). As we have observed, the kernel of $R_\epsilon$ is precisely $\mathcal K$. Moreover, by Green's theorem, the image of $R_\epsilon$ is orthogonal to $\mathcal K$.

We define an operator $L_{\epsilon}$ on $C^{\infty}(\oplus_{i=1}^kS^{1})$ as follows:
for any smooth function $f$ on $\oplus_{i=1}^k(S^1)$, let $u_{+}$ be the
harmonic function on $M_{\epsilon}$ with boundary values $f$ on $\Gamma_\epsilon$. Then let
$L_{\epsilon}(f)=(u_{+})|_{\Gamma_1}$. Note that $L_{\epsilon}$
is essentially the inverse of the operator $\mathcal E_{R(\epsilon)}$ from
\cite{w}; we use $L_{\epsilon}$ instead to avoid issues with the domain
of the extension operator $\mathcal E_{R(\epsilon)}$. We also observe that $L_\epsilon$ is the identity operator on $\mathcal K$. The following proposition is key:

\begin{proposition}\label{qinverse} For each $\epsilon<1$ and each $f\in\mathcal C\oplus L_0^2$, we also have $L_\epsilon f\in\mathcal C\oplus L_0^2$.
Moreover, for any $f\in C^{\infty}(\oplus_{i=1}^kS^{1})$ and any $\epsilon<1$,
\[||L_{\epsilon}f||_{L^{2}(\oplus_{i=1}^kS^{1})}\leq ||f||_{L^{2}(\oplus_{i=1}^kS^{1})}.\]
\end{proposition}
Note that the second statement allows us to extend $L_{\epsilon}$ by continuity to an
operator on $L^{2}(\oplus_{i=1}^kS^{1})$ which is bounded in operator norm by $1$. We call this extension $L_{\epsilon}$ as well.

\begin{proof} The proofs are identical to the proofs of Lemmas 3.3 and 3.4 in \cite{w}, and
involve integration and Green's theorem; we will not repeat them here. 
\end{proof}

For each nonzero $n\in\mathbb Z$ and $j$ between 1 and $k$, let $f_{n,j}$ be the function on $\oplus_{i=1}^k S^1$ which is equal to $e^{in\theta}$ on the $j$th component of $S^1$ and zero on all the other components. The collection of $f_{n,j}$ forms a basis of $L_0^2$. Again following \cite{w} and using the same notation, we define auxiliary operators $T_{\epsilon}$, $U_{\epsilon}^{\pm}$,
$\mathcal V$, $|\mathcal V|$, and $B$ on $L_0^2$ by 
\[T_{\epsilon}f_{n,j}=\frac{\epsilon^{n/\alpha_j}-\epsilon^{-n/\alpha_j}}
{\epsilon^{n/\alpha_j}+\epsilon^{-n/\alpha_j}}f_{n,j};\ 
\ U_{\epsilon}^{\pm}f_{n,j}=
\frac{1}{2}(\epsilon^{n/\alpha_j}\pm\epsilon^{-n/\alpha_j})f_{n,j};\]
\[\mathcal Vf_{n,j}=\frac{n}{\alpha_j}f_{n,j};\ \ 
|\mathcal V|f_{n,j}=\frac{|n|}{\alpha_j}f_{n,j},\]
extending by linearity to $L_0^2$. Note that each of these operators is invertible on $L_0^2$; the inverses may be written down explicitly. Additionally, let $f_{0,j}$ be the function which is 1 on the $j$th component of $S^1$ and zero on the others; these form a basis for $\mathcal K\oplus\mathcal C$. Finally, let $\bar B$ be the multiplication operator on $L^2(\oplus_{i=1}^k S^1)$ with $\bar Bf_{n,j}=b_jf_{n,j}$ for both zero and nonzero $n$; it restricts to an operator on $L_0^2$ but not to an operator on $\mathcal C$, so write $B=(\Pi_0+\Pi_{\mathcal C})\bar B(\Pi_0+\Pi_{\mathcal C})$.

We now compute $R_\epsilon$ on $\mathcal C\oplus L_0^2$ in terms of these auxiliary operators.
Let $P_{M_{\epsilon}}$ be the Dirichlet-to-Neumann
operator for $M_{\epsilon}$ and $P_{C_{\epsilon}}$ be the Dirichlet-to-Neumann
operator for $\oplus_{i=1}^k C_{\alpha_i,b_i\epsilon}$; then $R_{\epsilon}=P_{M_{\epsilon}}+
P_{C_{\epsilon}}$. As before, by Stokes' theorem, 
each of these operators also preserves the orthogonal
decomposition. Moreover, we may compute $\mathcal P_{C_{\epsilon}}$ directly: suppose that the boundary data is $f_{m,j}$.
Then by separation of variables, the harmonic extension to $C_{\alpha_j,b_j\epsilon}$ is precisely
$(r/b_j\epsilon)^{|m|/\alpha_j}e^{im\theta}$, and the harmonic extensions on the other components are zero. The inward-pointing normal derivative
at the boundary is $-\partial_{r}$, so we conclude that
\[P_{C_{\epsilon}}f_{m,j}=-\frac{|m|}{b_i\epsilon\alpha_j}f_{m,j}
=-\frac{1}{b_j\epsilon}|\mathcal V|f_{m,j}.\]
By linearity, we see that $\epsilon\mathcal P_{C_{\epsilon}}=-B^{-1}|\mathcal V|$
on $L_{0}^{2}$; note that $P_{C_\epsilon}=0$ on $\mathcal C$, since the harmonic extension of a constant function on $\Gamma_{\epsilon,i}$ to $C_{\alpha_i,b_i\epsilon}$ which is in the Friedrichs domain is itself constant. It remains to analyze $\mathcal P_{M_{\epsilon}}$.

\begin{lemma} We have the following four equations for $\epsilon P_{M_\epsilon}$:
\begin{equation}\label{goalone}\Pi_0(\epsilon\mathcal P_{M_{\epsilon}})\Pi_0=B^{-1}
\Pi_0(\mathcal V(\mathcal T_{\epsilon})^{-1}
-\mathcal V(\mathcal U^{-}_{\epsilon})^{-1}L_{\epsilon})\Pi_0,\end{equation}
\begin{equation}\label{goaltwo}\Pi_{\mathcal C}(\epsilon\mathcal P_{M_{\epsilon}})\Pi_0=B^{-1}
\Pi_{\mathcal C}(\frac{1}{\log(1/\epsilon)}L_{\epsilon})\Pi_0,\end{equation}
\begin{equation}\label{goalthree}\Pi_0(\epsilon\mathcal P_{M_{\epsilon}})\Pi_{\mathcal C}=B^{-1}\Pi_0
(-\mathcal V(\mathcal U^-_{\epsilon})^{-1}L_{\epsilon})\Pi_{\mathcal C},\end{equation}
\begin{equation}\label{goalfour}\Pi_{\mathcal C}(\epsilon\mathcal P_{M_{\epsilon}})\Pi_{\mathcal C}=B^{-1}
\Pi_{\mathcal C}(\frac{1}{\log(1/\epsilon)}(L_{\epsilon}-Id))\Pi_{\mathcal C}.\end{equation}
\end{lemma}

\begin{proof} First we prove (\ref{goalone}) and (\ref{goaltwo}). It is enough to prove the statements for the action of the operators on each $f_{m,j}$ for nonzero $m$. Let $u_{+}$ be the solution of the Dirichlet
problem on $M_{\epsilon}$ with boundary data equal to $f_{m,j}$. Since $u_{+}$ is harmonic, we may write that in a radius $3/2$-neighborhood of each conic point, where $r_i$ is the radial coordinate near $P_i$, that
\begin{equation}\label{harmext}
u_{+}(r,\theta)=\sum_{n\neq 0,\ 1\leq i\leq k}(a_{n,i}^{+}(\frac{r_i}{b_i\epsilon})^{-n/\alpha_i}
+a_{n,i}^{-}(\frac{r_i}{b_i\epsilon})^{-n/\alpha_i})f_{n,i}+\sum_{i=1}^k(a_{0,i}^++a_{0,i}^-\log\frac{r_i}{b_i\epsilon})f_{0,i}. \end{equation}
Since $u_{+}(b_i\epsilon,\theta)=f_{m,j}$, we see that $a_{n,i}^{+}+a_{n,i}^{-}=
\delta_{(n,i),(m,j)}$ for nonzero $m$, and that $a_{0,i}^+=0$, for each $i$. Therefore $u_{+}$ may be re-written
\begin{equation}\label{harmextsimp}
u_{+}(r,\theta)=(\frac{r_j}{b_j\epsilon})^{-m/\alpha_j}f_{m,j}+
\sum_{n\neq 0,\ 1\leq i\leq k}a_{n,i}^{+}((\frac{r_i}{b_i\epsilon})^{n/\alpha_i}-
(\frac{r_i}{b_i\epsilon})^{-n/\alpha_i})f_{n,i}+\sum_{i=1}^ka_{0,i}^-\log\frac{r_i}{b_i\epsilon}f_{0,i}.
\end{equation}
Now $P_{M_{\epsilon}}f_{m,j}=(\partial_{r_i}(u_{+}))|_{r_i=b_i\epsilon}$,
and $L_{\epsilon}f_{m,j}=u_{+}|_{r_i=b_i}$. Since the other operators
in (\ref{goalone}) are multiplication operators on each mode, and the projection operators single out the components $f_{m,j}$ with nonzero $m$, it is easy
to compute both sides of (\ref{goalone}) in terms of the $b_{n}^{+}$. We find
that both sides of (\ref{goalone}) are equal to
\[-\frac{m}{\alpha_jb_j}f_{m,j}+\sum_{n\neq 0,\ 1\leq i\leq k}\frac{2n}{\alpha_ib_i}a_{n,i}^{+}f_{n,i}.\]
This completes the proof of (\ref{goalone}). Similarly, from (\ref{harmextsimp}), we compute that both sides of (\ref{goaltwo}) are equal to
\[\sum_{i=1}^ka_{0,i}^-\frac{1}{b_i}f_{0,i},\] which verifies (\ref{goaltwo}).

On the other hand, to prove (\ref{goalthree}) and (\ref{goalfour}), it is enough to prove the statements for the action of the operators on each $f_{0,j}$. As before, let $u_+$ be the solution of the Dirichlet problem on $M_\epsilon$ with boundary data $f_{0,j}$; then we may still write (\ref{harmext}), and the boundary conditions allow us to simplify to:
\begin{equation}\label{harmextsimptwo}
u_{+}(r,\theta)=f_{0,j}+
\sum_{n\neq 0,\ 1\leq i\leq k}a_{n,i}^{+}((\frac{r_i}{b_i\epsilon})^{n/\alpha_i}-
(\frac{r_i}{b_i\epsilon})^{-n/\alpha_i})f_{n,i}+\sum_{i=1}^ka_{0,i}^-\log\frac{r_i}{b_i\epsilon}f_{0,i}.
\end{equation}
We may then compute, as before, that both sides of (\ref{goalthree}) are equal to 
\[\sum_{n\neq 0,\ 1\leq i\leq k}\frac{2n}{\alpha_ib_i}a_{n,i}^{+}f_{n,i}\]
and both sides of (\ref{goalfour}) are
\[\sum_{i=1}^ka_{0,i}^-\frac{1}{b_i}f_{0,i}.\]
This completes the proof of the lemma.
\end{proof}

Combining the lemma with the preceding remarks on $P_{C_\epsilon}$ gives a formula for $\epsilon R_\epsilon$ in terms of $L_\epsilon$ and the auxiliary operators; see Proposition \ref{formula} below.

\subsection{Determinant asymptotics}
In order to compute the asymptotics of the determinant of $\epsilon R_{\epsilon}$, we view $\epsilon R_\epsilon$ as the sum of a diagonal operator $A_\epsilon$ and a 'small perturbation' $K_\epsilon$. The following proposition is immediate from the previous lemma and discussion:
\begin{proposition}\label{formula} Let $A_\epsilon$ and $K_\epsilon$ be operators on $\mathcal C\oplus L_0^2$ given by:
\[BA_\epsilon=-2\Pi_0|\mathcal V|\Pi_0-\Pi_{\mathcal C}\frac{1}{\log(1/\epsilon)}\Pi_{\mathcal C},\]
\[BK_\epsilon=\Pi_0(|\mathcal V|+\mathcal V(\mathcal T_{\epsilon})^{-1}-\mathcal V(\mathcal U_\epsilon^-)^{-1}L_\epsilon)\Pi_0+\Pi_{\mathcal C}(\frac{1}{\log(1/\epsilon)}L_{\epsilon})\Pi_0\]
\[+\Pi_0
(-\mathcal V(\mathcal U^-_{\epsilon})^{-1}L_{\epsilon})\Pi_{\mathcal C}+\Pi_{\mathcal C}(\frac{1}{\log(1/\epsilon)}L_\epsilon)\Pi_{\mathcal C}.\]
Then $\epsilon R_\epsilon=A_\epsilon+K_\epsilon$.
\end{proposition}

The key to the proof of Theorem \ref{asymplogdet} is that the asymptotics of the determinant of $\epsilon R_{\epsilon}$ are closely related to those of the determinant of $A_{\epsilon}$, which are easy to compute. The following lemma, whose proof is deferred until the next section, provides the necessary comparison:

\begin{lemma}\label{difference}As $\epsilon\rightarrow 0$,
\[\log\det\epsilon R_\epsilon-\log\det A_\epsilon=o(1).\]
Moreover, when $k=1$, $b_1=1$, and $\epsilon\leq 2^{-\alpha_1}$,
\[|\log\det\epsilon R_\epsilon-\log\det A_\epsilon|\leq 6\epsilon^{1/\alpha_1}.\]
\end{lemma}

We now compute the asymptotics of the determinant of $A_{\epsilon}$. The eigenvalues of $2B^{-1}|\mathcal V|$ on $L_0^2$ are precisely $2m/(b_i\alpha_i)$
for each $m$ in $\mathbb N$ and each $i$ between 1 and $k$, each with multiplicity 2.
The zeta function for $2|\mathcal V|$ is therefore
\[\sum_{i=1}^{k}2(\frac{2}{b_i\alpha_i})^{-s}\zeta_{Riem}(s).\]
Taking the derivative at zero and using the special values of the Riemann
zeta function, we compute that the log-determinant of $A_\epsilon$ on $L_0^2$ is $\sum_{i=1}^k\log\pi b_i\alpha_i$. Moreover, to compute the log-determinant of $A_\epsilon$ on $\mathcal C$, first let $\hat B$ be the $(k-1)$-dimensional matrix $\Pi_C B\Pi_C$; it is easy to see that $\hat B$ is invertible with all positive eigenvalues. Then the log-determinant of $A_\epsilon$ on $\mathcal C$ (note that we take absolute values of the eigenvalues first) is just \[\log\det\hat B^{-1} + (k-1)\log(-(\log\epsilon)^{-1})=\log\det\hat B^{-1} -(k-1)\log\log(1/\epsilon).\]

Next, we compute $\log\det\hat B$. For each $i$ between 1 and $k-1$, we let $g_i=f_{0,i}-f_{0,k}$. Then we have
\[\hat Bg_i=\Pi_{\mathcal C}(b_if_{0,i}-b_kf_{0,k})=b_ig_i-\frac{b_i-b_k}{k}\sum_{j=i}^kg_j.\]
So we must take the log of the determinant of a matrix whose $(i,j)$ entry is \[\frac{b_k-b_j}{k}+b_j\delta_{ij}.\]
The determinant must be a polynomial in the $\{b_i\}$ of joint degree $k-1$. Since the numbering was arbitrary, it must be symmetric. Moreover, by inspection, it has degree at most 1 in each of the $b_i$ between 1 and $k-1$, and hence by symmetry also in $b_k$. These observations force the polynomial to be $\beta$ times a constant which depends only on $k$. But if each $b_i=1$, then $\hat B=Id$ and the determinant is 1; therefore, the constant is 1, and $\log\det\hat B=\log\beta$. We have now shown that
\begin{equation}\label{almost}\log\det A_{\epsilon}=-(k-1)\log\log(1/\epsilon)+\sum_{i=1}^k\log\pi b_i\alpha_i-\log\beta.\end{equation}

Finally, we need to go from $\log\det\epsilon R_{\epsilon}$ to
$\log\det R_{\epsilon}$. The argument is similar to Proposition \ref{scaling};
since the eigenvalues scale by $\epsilon$ in this case, as opposed to
$\epsilon^{-2}$, one can easily compute
\begin{equation}\label{factorepsilon}\log\det\epsilon R_{\epsilon}=\log\det R_{\epsilon}+\zeta_{R_{\epsilon}}(0)
\log\epsilon.\end{equation}
We must still compute $\zeta_{R_{\epsilon}}(0)$, but it turns out to be simple:
\begin{proposition} For any $\epsilon$, $\zeta_{R_{\epsilon}}(0)=-1$.
\end{proposition}

\begin{proof} Fix $\epsilon$. The result follows by using homogeneity; we examine (\ref{bfkgluing}) and note that the same
gluing formula applies to $(M,\delta^{2}g)$ instead of $(M,g)$, for any $\delta$.
We then compute the variation of each of the terms.
For each of the $\log\det\Delta$ terms, we may apply Proposition \ref{scaling}.
$\log V$ changes by $2\log\delta$ and $\log(l(\partial M_\epsilon))$ changes by
$\log\delta$. Since the eigenvalues of $R_{\epsilon}$ scale by $\delta$,
$\log\det R_{\epsilon}$ changes by $-\zeta_{R}(0)\log\delta$. Putting
everything together, we have:
\[-2\zeta_{M}(0)\log\delta=-2\zeta_{M_{\epsilon}}(0)\log\delta
-2\sum_{i=1}^k\zeta_{C_{\alpha_i,b_i\epsilon}}(0)\log\delta+2\log\delta
-\log\delta-\zeta_{R_\epsilon}(0)\log\delta.\]
Since $\delta$ is arbitrary, we see that
\begin{equation}\label{zetavaluetwo}\zeta_{R_{\epsilon}}(0)=2\zeta_{M}(0)
-2\zeta_{M_{\epsilon}}(0)
-2\sum_{i=1}^k\zeta_{C_{\alpha_i,b_i\epsilon}}(0)-1.\end{equation}

Recall that for any surface $\Omega$ with or without boundary, $\zeta_{\Omega}(0)=\frac{1}{6}\chi(\Omega)$ \cite{r}. When $\Omega$
has conical singularities, work of Cheeger \cite{ch2} shows
that the same formula holds, with an additional contribution from
each conic singularity depending only on the cone angle $\alpha$.
Applying these formulas to (\ref{zetavaluetwo}), we see that all the
contributions from the conic singularities cancel, and we are left with
\[\zeta_{R_{\epsilon}}(0)=\frac{1}{3}(\chi(M)-\chi(M_{\epsilon})-\sum_{i=1}^k\chi(C_{\alpha_i,b_i\epsilon}))-1.\]
However, $\chi(M)=\chi(M_{\epsilon})+\sum_{i=1}^k\chi(C_{\alpha_i,b_i\epsilon})$ by the definition of the Euler characteristic, and hence $\zeta_{R_{\epsilon}}(0)=-1$.
\end{proof}
Combining this proposition with (\ref{almost}), (\ref{factorepsilon}), and Lemma \ref{difference} completes the proof
of Theorem \ref{asymplogdet}, and thus also finishes the proof of
Theorem \ref{truncated} and \ref{bettererror}.

\subsection{Proof of Lemma \ref{difference}}
\begin{proof}
The proof is based on similar arguments in \cite{l,w}. Since $\epsilon R_\epsilon=A_\epsilon+K_\epsilon$, we have formally that
\[\log\det\epsilon R_\epsilon-\log\det A_\epsilon=\int_0^1\frac{d}{dt}\log\det(A_\epsilon+tK_\epsilon)\ dt=\int_0^1Tr ((A_\epsilon+tK_\epsilon)^{-1}K_\epsilon)\ dt,\]
and hence, whenever $||K_{\epsilon}A_{\epsilon}^{-1}||$ has norm less than 1,
\[|\log\det\epsilon R_\epsilon-\log\det A_\epsilon|\leq \int_0^1||A_{\epsilon}^{-1}||\cdot||(Id + tK_{\epsilon}A_{\epsilon}^{-1})^{-1}||\cdot Tr |K_{\epsilon}|\ dt\]
\[\leq \int_0^1||A_{\epsilon}^{-1}||\cdot\frac{1}{1-||tK_{\epsilon}A_{\epsilon}^{-1}||}\cdot Tr |K_{\epsilon}|\ dt\]
\begin{equation}\label{comeback}\leq ||A_{\epsilon}^{-1}||\cdot\frac{1}{1-||K_{\epsilon}A_{\epsilon}^{-1}||}\cdot Tr |K_{\epsilon}|.\end{equation}
In order for this calculation to be justified, $K_{\epsilon}$ must be trace class and $||K_{\epsilon}A_{\epsilon}^{-1}||$ must have norm less than 1. As we analyze (\ref{comeback}), we will prove that both of these conditions hold for sufficiently small $\epsilon$. The first step in that analysis is an observation which is an immediate consequence of the definition of $A_{\epsilon}$:
\begin{proposition}\label{aepsilon} $A_{\epsilon}$ is invertible, and the norm of the inverse is bounded by $||B^{-1}||\log (1/\epsilon)$ for sufficiently small $\epsilon$. On the other hand, if $k=1$ and $b_1=1$, then $A_{\epsilon}=-2|\mathcal V|$ and hence $||A_{\epsilon}^{-1}||=\alpha_1/2.$ \end{proposition}

Next, we show that $K_{\epsilon}$ is trace class and estimate its trace norm.

\begin{proposition}\label{ktraceclass} $K_\epsilon$ is trace class for all $\epsilon\in (0,1)$, and
\[\lim_{\epsilon\rightarrow 0}\log(1/\epsilon)\ Tr |K_\epsilon|=0.\]
Moreover, if $k=1$, $b_1=1$, and $\epsilon\leq 2^{-\alpha_1}$,
\[Tr |K_{\epsilon}|\leq \frac{12}{\alpha_1}\epsilon^{1/\alpha_1}.\]
\end{proposition}

\begin{proof} Since $B$ and $B^{-1}$ both have bounded norm, it suffices to analyze $BK_\epsilon$ instead. We compute the trace directly. First analyze the trace of $BK_\epsilon$ on $L_0^2$. As in \cite{w}, the operator $|\mathcal V|+\mathcal V(\mathcal T_\epsilon)^{-1}$ is diagonal on $L_0^2$, with eigenvalues
\begin{equation}\label{evalstwo}
\frac{m}{\alpha_j}\frac{\epsilon^{m/\alpha_j}+\epsilon^{-m/\alpha_j}}
{\epsilon^{m/\alpha_j}-\epsilon^{-m/\alpha_j}}+\frac{|m|}{\alpha_j}=
\frac{2|m|}{\alpha_j}\frac{1}{1-\epsilon^{-2|m|/\alpha_j}}
\end{equation}
for each $j$ between 1 and $k$. The sum over $m$ and $j$ of the absolute value of these eigenvalues converges for all $\epsilon<1$. Moreover, let $\alpha=\min_{j=1}^k\{\alpha_j\}$; this sum times $\epsilon^{-1/\alpha}$ is bounded by
\[\sum_{m=1}^{\infty}\sum_{j=1}^k\frac{4m}{\alpha_j}\epsilon^{\frac{2m-1}{a_j}}.\]
On the other hand, the operator $\mathcal V(\mathcal U_\epsilon^-)^{-1}$ is diagonal on $L_0^2$ with eigenvalues
\begin{equation}\label{evalsone}
\frac{m}{\alpha_j(\epsilon^{m/\alpha_j}-\epsilon^{-m/\alpha_j})};\end{equation}
the sum of the absolute values again converges for all $\epsilon<1$. After multiplying by $\epsilon^{-1/\alpha}$, it is bounded by
\[\sum_{m=1}^{\infty}\sum_{j=1}^k\frac{2m}{a_j}\epsilon^{\frac{m-1}{a_j}}.\]
Since $||L_{\epsilon}||\leq 1$, we conclude that $\epsilon^{-1/\alpha}$ times the trace of $|K_{\epsilon}|$ on $L_0^2$ is bounded by
\[\sum_{m=1}^{\infty}\sum_{j=1}^k(\frac{4m}{\alpha_j}\epsilon^{\frac{2m-1}{a_j}}+\frac{2m}{a_j}\epsilon^{\frac{m-1}{a_j}})\leq\frac{6k}{\alpha}\sum_{m=1}^{\infty}m\epsilon^{\frac{m-1}{a_j}}.\]
When $\epsilon^{1/\alpha}<1/2$, this is bounded by $12k/\alpha$. Since $\mathcal C=\emptyset$ when $k=1$, the second claim in Proposition \ref{ktraceclass} follows immediately, and we also see that $\log(1/\epsilon)$ times the trace of $|K_{\epsilon}|$ on $L_0^2$ approaches zero in general.

It remains to analyze $BK_\epsilon$ on $\mathcal C$; we analyze
\begin{equation}\label{constpart}\log(1/\epsilon)\Pi_{\mathcal C}BK_\epsilon\Pi_{\mathcal C}=\Pi_{\mathcal C}L_\epsilon\Pi_{\mathcal C}.\end{equation}
The operator (\ref{constpart}) is finite-dimensional and hence trace class, and the trace norm is bounded by $(k-1)||\Pi_{\mathcal C}L_\epsilon\Pi_{\mathcal C}||$. It suffices to show that $||\Pi_{\mathcal C}L_\epsilon\Pi_{\mathcal C}||$ goes to zero as $\epsilon$ goes to zero. Suppose not. Then there is a sequence $g_i$ of functions in $\mathcal C$, of norm one in $L^2(\oplus_{i=1}^k S^1)$, and a sequence $\epsilon_i\rightarrow 0$ where $\Pi_{\mathcal C}L_{\epsilon_i} g_{i}$ has $L^2$-norm bounded below. For each $i$, let $u_{+,i}$ be the harmonic function on $M_{\epsilon_i}$ with boundary data $g_{i}$. Since $g_{i}$ has norm one in $L^2(\oplus_{i=1}^k S^1)$ and is piecewise constant, the $L^{\infty}$ norm of $g_{i}$ is bounded by $\frac{1}{2\pi}<1$ for all $i$. By the maximum principle, $u_{+,i}$ is also bounded by 1 for each $i$. Therefore, as $i\rightarrow\infty$, the Arzela-Ascoli theorem implies that a subsequence of $u_{+,i}$ converges uniformly on compact subsets of $M$ away from the conic tip to a limit function $u$. By standard elliptic theory, the limit $u$ is itself harmonic on $M$, and is obviously bounded by 1 (and therefore is in the Friedrichs domain at each conic point). Therefore, $u$ itself is constant, so $h=u|_{\partial M_1}$ is globally constant. However, by construction, passing to the Arzela-Ascoli subsequence, $L_{\epsilon_i}g_i\rightarrow h$ uniformly as $i\rightarrow\infty$. So $\Pi_{\mathcal C}L_{\epsilon_i}g_i\rightarrow\Pi_{\mathcal C} h$ uniformly (and hence strongly in $L^2$), but $\Pi_{\mathcal C} h=0$. This is a contradiction which completes the proof of the proposition.
\end{proof}

Finally, we prove the necessary bounds for $||K_\epsilon A_\epsilon^{-1}||$:

\begin{proposition}\label{pertinvert} The norm $||K_\epsilon A_\epsilon^{-1}||$ approaches zero as $\epsilon\rightarrow 0$. Moreover, when $k=1$, $b_1=1$, and $\epsilon\leq 2^{-\alpha_1}$, $||K_{\epsilon} A_\epsilon^{-1}||\leq 1/2$.
\end{proposition}

\begin{proof} To prove the first claim, observe from the definition of $A_\epsilon$ and $K_\epsilon$ that
\[BK_\epsilon A_\epsilon^{-1}B^{-1}=-\frac{1}{2}\Pi_0(|\mathcal V|+\mathcal V(\mathcal T_{\epsilon})^{-1}-\mathcal V(\mathcal U_\epsilon^-)^{-1}L_\epsilon)|\mathcal V|^{-1}\Pi_0-\frac{1}{2}\Pi_{\mathcal C}(\frac{1}{\log(1/\epsilon)}L_{\epsilon})|\mathcal V|^{-1}\Pi_0\]
\begin{equation}\label{thingone}+\log(1/\epsilon)\Pi_0
(-\mathcal V(\mathcal U^-_{\epsilon})^{-1}L_{\epsilon})\Pi_{\mathcal C}+\Pi_{\mathcal C}L_\epsilon\Pi_{\mathcal C}.\end{equation}
We analyze each term in (\ref{thingone}) separately and show that its norm goes to zero as $\epsilon$ goes to zero. For the first term, note that $|||\mathcal V|^{-1}||$ is bounded by $\alpha=\max_i\{\alpha_i\}$ and $||L_{\epsilon}||\leq 1$. The remaining operators are diagonalized by the $f_{m,j}$, and by the eigenvalue calculations in the proof of Proposition \ref{ktraceclass}, all the eigenvalues go to zero as $\epsilon$ goes to zero; this is more than sufficient. Similarly, the norm of the second term is bounded by $(\alpha/2)(\log(1/\epsilon))^{-1}$, which goes to zero as $\epsilon$ goes to zero.
We again use $||L_\epsilon||<1$ and the eigenvalue calculations from the previous proposition to analyze the third term; the eigenvalues of $\mathcal V(\mathcal U_\epsilon^-)^{-1}$ go to zero even when multiplied by $\log(1/\epsilon)$, which gives the norm bound. Finally, we have already done the necessary analysis for the fourth term in the proof of Proposition \ref{ktraceclass}. This completes the proof of the claim and hence of the general case of the proposition.

As for the $k=1$ and $b_1=1$ bound: in this case, $A_{\epsilon}^{-1}$ is given by multiplication by $-\frac{\alpha_1}{2|m|}$ on each mode. Since $||L_{\epsilon}||\leq 1$, the norm $||K_{\epsilon}A_{\epsilon}^{-1}||$ is bounded by the maximum of $\frac{\alpha_1}{2|m|}$ times (\ref{evalstwo}) plus the maximum of $\frac{\alpha_1}{2|m|}$ times (\ref{evalsone}). Using similar analysis as in the proof of Proposition \ref{ktraceclass}, we see that this is at most
\[\epsilon^{2/\alpha_j}+\frac{1}{2}\epsilon^{1/\alpha_j},\]
which is certainly bounded by $1/2$ when $\epsilon\leq 2^{-\alpha_j}$ (a sharp bound is not necessary). This completes the proof.
\end{proof}

It is now an immediate consequence of the three propositions that (\ref{comeback}) approaches zero as $\epsilon\rightarrow 0$. Moreover, in the $k=1$, $b_1=1$, $\epsilon\leq 2^{-\alpha_1}$ case, we see that (\ref{comeback}) is bounded by $6\epsilon^{1/\alpha_1}$. This completes the proof of Lemma \ref{difference}.
\end{proof}

\end{document}